\documentclass[12pt,reqno]{amsart}

\title[]{Some properties of path measures}
\author{Christian L\'eonard}

\usepackage{amssymb, amsmath, amsfonts, latexsym, enumerate}
\usepackage[usenames,dvipsnames]{pstricks}
\usepackage{epsfig}

\newtheorem{theorem}{Theorem}
\newtheorem{lemma}[theorem]{Lemma}
\newtheorem{proposition}[theorem]{Proposition}
\newtheorem{corollary}[theorem]{Corollary}

\newtheorem{definition}[theorem]{Definition}
\newtheorem{definitions}[theorem]{Definitions}

\newtheorem{assumption}[theorem]{Assumption}

\theoremstyle{remark}
\newtheorem{remark}[theorem]{Remark}
\newtheorem{remarks}[theorem]{Remarks}
\newtheorem{example}[theorem]{Example}
\newtheorem{examples}[theorem]{Examples}

\numberwithin{theorem}{section}
\numberwithin{equation}{section}

\newcommand{\RR}{\mathbb{R}}
\newcommand{\Rn}{\mathbb{R}^n}

\newcommand{\1}{\mathbf{1}}

\newcommand{\ttimes}{\!\times\!}

\newcommand\pf{_{\#}}

\renewcommand{\ae}{\textrm{-a.e.}}

    \DeclareMathOperator{\dom}{dom}

\newcommand{\boulette}[1]{$\bullet$\ Proof of #1.}
\newcommand{\Boulette}[1]{\par\medskip\noindent $\bullet$\ Proof of #1.}

\newcommand\seq[2]{(#1_#2)_{#2\ge1}}

\newcommand\Lim[1]{\lim_{#1\rightarrow\infty}}
\newcommand\Liminf[1]{\liminf_{#1\rightarrow\infty}}

\newcommand{\lsc}{lower semicontinuous}

\newcommand{\cadlag}{c\`adl\`ag}

\newcommand\XX{\mathcal{X}}
\newcommand\XXX{\XX^2}
\newcommand\PX{\mathrm{P}(\XX)}

\newcommand\PXX{\mathrm{P}(\XXX)}
\newcommand\MX{\mathrm{M}_+(\XX)}

\newcommand\PO{\mathrm{P}(\Omega)}
\newcommand\MO{\mathrm{M}_+(\Omega)}

\newcommand\OO{\Omega}

\newcommand\ii{{[0,1]}}

\newcommand\IX{\int_{\XX}}
\newcommand\IXX{\int_{\XXX}}
\newcommand\IO{\int_\Omega}

\newcommand\Ph{\widehat{P}}

\newcommand\AAA[2]{\mathcal{A}_{[#1,#2]}}

\begin{document}


 \address{Modal-X. Universit\'e Paris Ouest. B\^at.\! G, 200 av. de la R\'epublique. 92001 Nanterre, France}
 \email{christian.leonard@u-paris10.fr}
 \keywords{Unbounded measure, conditional expectation, relative entropy, stochastic processes, Schr\"odinger problem}
 \subjclass[2010]{28A50,60J25}

\begin{abstract} 
We call any measure on a path space, a \emph{path measure}.
Some notions about path measures which appear naturally  when solving the Schr\"odinger problem  are presented and worked out in detail. 
\end{abstract}

\maketitle 
\tableofcontents


\section*{Introduction}

We call any measure on a path space, a \emph{path measure}.
Some notions about path measures which appear naturally  when solving the Schr\"odinger problem (see \eqref{eq-03} below and \cite{Leo12e}) are presented and worked out in detail.

\subsection*{Aim of this article}

This paper is about three separate items :
\begin{enumerate}
\item
Disintegration of  an unbounded  measure;
\item
Basic properties of the relative entropy with respect to an unbounded measure;
\item
Positive integration with respect to a Markov measure.
\end{enumerate}

Although items (1) and (2) are mainly about general unbounded measures, we are motivated by their applications to path measures. 
\\
In particular, it is shown that when $Q$ is an unbounded path measure, some restriction must be imposed on $Q$ for being able to consider conditional expectations such as $Q(\cdot|X_t).$ This is the content of the notion of \emph{conditionable path measure} which is introduced at Definition \ref{def-06}. 
\\
Some care is also required when working with the relative entropy with respect to an unbounded reference measure. We also give a detailed proof of the additive  property of the relative entropy at Theorem \ref{res-02}. Indeed, we didn't find in the literature a complete proof of this  well known result.

\subsection*{Some notation}

Let $\XX$ be a Polish state space furnished with the corresponding Borel $\sigma$-field and $\OO=D(\ii,\XX)$ the space of all \cadlag\ (right-continuous and left-limited) paths from the unit time interval $\ii$ to $\XX.$ Depending on the context, we may only consider $\OO=C(\ii,\XX)$, the space of all continuous paths. As usual, the $\sigma$-field on $\OO$ is generated by the canonical process
\begin{equation*}
X_t(\omega):= \omega_t\in\XX,\quad \omega=(\omega_s)_{0\le s\le1}\in\OO,\ 0\le t\le1.
\end{equation*}

We write $\mathrm{M}_+(Y)$ for the set of all nonnegative measures on a space $Y$, and $\mathrm{P}(Y)$ for the subset of all  probability measures. Let $Q\in \mathrm{M}_+(Y),$ the push-forward of $Q$ by the measurable mapping $\phi:Y\to\XX$ is denoted by $\phi\pf Q$ or $Q_\phi\in\MX.$
\\
Any positive measure $Q\in\MO$ on the path space $\OO$ is called a \emph{path measure}. For any subset $\mathcal{T}\subset\ii,$ we denote $X _{\mathcal{T}}=(X_t)_{t\in \mathcal{T}}$ and $Q _{\mathcal{T}}=(X _{\mathcal{T}})\pf Q=Q(X _{\mathcal{T}}\in\cdot)\in \mathrm{M}_+(\OO _{\mathcal{T}})$ the push-forward of $Q$ by $X _{\mathcal{T}}$ on the set of positive measures on the restriction $\OO _{\mathcal{T}}$ of $\OO$ to $\mathcal{T}.$ In particular, for each $0\le t\le1,$ $Q_t=Q(X_t\in\cdot)\in\MX.$

\subsection*{Motivation}

Take a reference path measure $R\in\MO$ and consider the problem
\begin{equation}\label{eq-03}
H(P|R)\to \textrm{min};\qquad P\in\PO:P_0=\mu_0, P_1=\mu_1
\end{equation}
of minimizing the relative entropy $$H(P|R):=\IO\log \left(\frac{dP}{dR}\right) \,dP\in (-\infty,\infty]$$ of $P\in\PO$ with respect to $R\in\MO$, among all the  path probability measures $P\in\PO$ such that the initial and final marginals $P_0$ and $P_1$ are asked to equal respectively two prescribed probability measures $\mu_0$ and $\mu_1\in\PX$ on the state space $\XX.$  
This entropy minimization problem is called the \emph{Schr\"odinger problem}. It is described in the author's survey paper \cite{Leo12e} where it is exemplified with $R$ a reversible Markov process, for instance the \emph{reversible Brownian motion} on $\Rn.$ 
\\
If one wants to describe  the reversible Brownian motion on $\Rn$ as a measure on the path space $\OO=C(\ii,\Rn),$ one has to consider an \emph{unbounded measure}. Indeed, its reversing measure  is Lebesgue measure (or any of its positive multiple), and its ``law" is
\begin{equation*}
R=\int _{\Rn}\mathcal{W}_x(\cdot)\,dx\in\MO,
\end{equation*}
where $\mathcal{W}_x\in\PO$ stands for the Wiener measure with starting position $x\in\Rn.$ Obviously, this path measure  has the same unbounded mass as Lebesgue measure.
More generally, any path measure $Q\in\MO$ has the same mass as its time-marginal measures $Q_t\in\MX$ for all $t\in\ii.$ In particular, any reversible path measure in $\MO$ with an unbounded reversing measure in $\MX,$ is also unbounded. 
\\
In connection with the Schr\"odinger problem, the notion of $(f,g)$-transform of a possibly unbounded Markov measure $R$ is introduced in \cite{Leo12e}. It is defined by
\begin{equation}\label{eq-02}
P=f(X_0)g(X_1)\,R\in\PO
\end{equation}
where $f$ and $g$ are measurable nonnegative functions such that $E_R(f(X_0)g(X_1))=1.$
It is a time-symmetric extension of the usual Doob $h$-transform. It appears that the product form  of the Radon-Nikodym derivative $f(X_0)g(X_1)$ implies that $P$ is the solution to the Schr\"odinger problem with the correct prescribed marginals $\mu_0$ and $\mu_1$ which are given by
\begin{equation}\label{eq-01}
\left\{\begin{array}{lcl}
\mu_0(dx)&=&f(x)E_R(g(X_1)\mid X_0=x)\, R_0(dx),\\
\mu_1(dy)&=& E_R(f(X_0)\mid X_1=y)g(y)\, R_1(dy).
\end{array}\right.
\end{equation}

\subsection*{Disintegration of  an unbounded path measure}

One has to be careful when saying that the reversible Brownian motion $R\in\MO$ is Markov. Of course, this means that for all $0\le t\le1,$ $E_R(X _{[t,1]}\in\cdot\mid X _{[0,t]})=E_R(X _{[t,1]}\in\cdot\mid X_t).$ Similarly, we wrote \eqref{eq-01} without hesitating.  But the problem is to define properly the conditional expectation with respect to an \emph{unbounded} measure.  This will be the purpose of Section \ref{sec-disint} where  extensions of the conditional expectation are considered and a definition of the Markov property for an unbounded path measure is given. The general theory of conditional expectation is recalled at the appendix Section \ref{sec-CE} to emphasize the role of $\sigma$-finiteness.
 
\subsection*{Relative entropy with respect to an unbounded measure}

The relative entropy with respect to a probability measure is well-known. But once we have an unbounded path measure at hand, what about the relative entropy with respect to an unbounded measure and its additive  property? This is the subject of Section \ref{sec-entropy}.

\subsection*{Positive integration with respect to a Markov measure}

It is assumed in the $(f,g)$-transform formula \eqref{eq-02} that $E_R(f(X_0)g(X_1))<\infty$ with $f,g\ge0,$ while the conditional expectations $E_R(f(X_0)\mid X_1)$ and $E_R(g(X_1)\mid X_0)$ appear at \eqref{eq-01}. But the assumption that  $f(X_0)g(X_1)$ is $R$-integrable doesn't ensure, in general, that $f(X_0)$ and $g(X_1)$ are separately $R$-integrable; which is a prerequisite for defining properly the conditional expectations $E_R(f(X_0)\mid X_1)$ and $E_R(g(X_1)\mid X_0)$. However, we need a general setting for the conditional expectations in \eqref{eq-01} to be meaningful. This will be presented at Section \ref{sec-positive} where we take advantage of the positivity of the functions $f$ and $g$.

\section{Disintegration of  an unbounded path measure}\label{sec-disint}

We often need the following notion which is a little more restrictive than the absolute continuity, but which matches with it whenever the measures are $\sigma$-finite.

\begin{definition} Let $R$ and $Q\in\MO$ be two positive measures on some measurable space $\Omega.$
One says that \emph{$Q$ admits a density with respect to $R$} if there exists a measurable function $\theta:\Omega\to[0,\infty)$ which verifies
\begin{equation*}
    \IO f\,dQ=\IO f\theta \,dR\in [0,\infty],\quad \forall f\ge0 \textrm{ measurable}.
\end{equation*}
    We write this relation
    $$
Q\prec R
    $$
    and we denote
    $$
\theta:=\frac{dQ}{dR}
    $$
    which is called the  \emph{Radon-Nikodym derivative} of $Q$ with respect to $R.$
\end{definition}

Thanks to the monotone convergence theorem, it is easy to check that if $R$ is $\sigma$-finite and
$\theta:\Omega\to[0,\infty)$ is a nonnegative measurable function, then
$$\theta R(A):=\int_A\theta\,dR,\quad A\in \mathcal{A},$$ defines a positive measure on the $\sigma$-field $\mathcal{A}.$

\begin{proposition}\label{res-LZ03}
Let $R$ and $Q$ be two positive measures. Suppose that $R$ is
$\sigma$-finite. The following assertions are equivalent:
\begin{enumerate}
    \item[(a)] $Q\prec R$
    \item[(b)] $Q$ is $\sigma$-finite and $Q\ll R.$
\end{enumerate}
\end{proposition}

\begin{proof}
The implication $(b) \Rightarrow (a)$ is Radon-Nikodym Theorem \ref{res-LZ33}. Let us show its converse  $(a)\Rightarrow (b).$
The absolute continuity $Q\ll R$ is straightforward. Let us prove that $Q$ is $\sigma$-finite. Let $\seq An$ be a  $\sigma$-finite partition of $R.$ Define for all $k\ge1,$ $B_k=\{k-1\le dQ/dR< k\}$.  The sequence $\seq Bk$ is also a measurable partition. Hence, $(A_n\cap B_k)_{n,k\ge1}$ is a countable measurable partition. On the other hand, for any $(n,k),$ $Q(A_n\cap B_k)=E_R(\1_{A_n\cap B_k}\,dQ/dR)\le k
 R(A_n)<\infty.$ Therefore $(A_n\cap B_k)_{n,k\ge1}$ is a  $\sigma$-finite partition of $Q.$
\end{proof}

Let $Q,R\in\MO$ be two (possibly unbounded) positive measures on $\OO$.  Let $\phi:\OO\to\XX$ be a measurable mapping from $\OO$ to a Polish (separable, complete metric) space $\XX$ equipped with its Borel $\sigma$-field.
Although $Q\ll R$ implies that $Q_\phi\ll R_\phi,$ in general we do not have 
$Q_\phi\prec R_\phi$ when $Q\prec R,$ as the following example shows;

\begin{example}\label{exp-02}   The measure $R$ is the uniform probability measure on $\Omega=[0,1]\times
[0,1],$ $Q$ is defined by $Q(dxdy)=1/y\, R(dxdy)$ and we denote the canonical projections by
$\phi_X(x,y)=x,$ $\phi_Y(x,y)=y,$ $(x,y)\in\Omega.$  We observe that on the one hand $R,$ $Q$ and
$R_{\phi_X}(dx)=\mathrm{Leb}(dx)=dx$ are $\sigma$-finite, but on the other hand, $Q_{\phi_X}$ is defined by
$Q_{\phi_X}(A)=\left\{\begin{array}{ll}
  0 & \textrm{if } \mathrm{Leb}(A)=0\\
  +\infty & \textrm{otherwise} \\
\end{array}\right..$ We have $Q_{\phi_X}\ll R_{\phi_X},$ but $Q_{\phi_X}$
is not $\sigma$-finite. We also see that $Q_{\phi_Y}(dy)=1/y\,dy$
is $\sigma$-finite.
\end{example}

\subsection*{An extension of the conditional expectation}

To extend easily results about conditional expectation with respect to a bounded measure (in particular Propositions \ref{res-LZ06} and \ref{res-LZ04}) to a $\sigma$-finite measure, it is useful to rely on the following preliminary result.

\begin{lemma}\label{res-LZ08} Let us assume that $R_\phi$ is $\sigma$-finite.
\begin{enumerate}[(a)]
    \item Let $\gamma:\XX\to (0,1]$ be a measurable function such that $\gamma R_\phi$ is a bounded measure. Then, $L^1(R)\subset L^1(\gamma(\phi) R)$ and for any $f\in
L^1(R),$
    $ 
    E_R(f\mid\phi)=E_{\gamma(\phi) R}(f\mid\phi), R\ae
    $ 
    \item There exists a function $\gamma\in
L^1(R_\phi)$ such that $0<\gamma\le1, R_\phi\ae$ In particular,
the measure $\gamma(\phi) R$ is bounded and equivalent to $R,$ i.e. for any measurable subset $A,$ $R(A)=0\iff [\gamma(\phi) R](A)=0.$
     \item Let $Q$ be another positive measure on $\Omega$ such that $Q_\phi$ is $\sigma$-finite.
Then, there exists a function $\gamma\in L^1(R_\phi+Q_\phi)$ such that $0<\gamma\le1, (R_\phi+Q_\phi)\ae$ In particular, the measures  $\gamma(\phi) R$ and $\gamma(\phi) Q$ are bounded and respectively equivalent to  $R$ and $Q.$
\end{enumerate}
\end{lemma}

\begin{proof}
    \boulette{(a)}
Denote $B_\phi$ the space of all $\mathcal{A}(\phi)$-measurable and bounded functions and $\gamma B_\phi:=\{h:h/\gamma(\phi)\in
B_\phi\}\subset B_\phi.$ For all $f\in L^1(R)$ and
$h\in \gamma B_\phi,$
\begin{eqnarray*}
  \IO hf\,dR
  &=& \IO \frac h{\gamma(\phi)} f\gamma(\phi)\,dR \\
  &=& \IO \frac h{\gamma(\phi)} E_{\gamma(\phi) R}(f\mid\phi)\, d(\gamma(\phi) R) \\
  &=& \IO h E_{\gamma(\phi) R}(f\mid\phi)\,dR.
\end{eqnarray*}
On the other hand, $\IO hf\,dR=\IO h E_R(f\mid\phi)\,dR$ so that
$$
\IO h E_R(f\mid\phi)\,dR_{\mathcal{A}(\phi)}=\IO h E_{\gamma(\phi)
R}(f\mid\phi)\,dR_{\mathcal{A}(\phi)}, \quad\forall h\in\gamma B_\phi.
$$
In other words, the measures $E_R(f\mid\phi)R_{\mathcal{A}(\phi)}$ and $E_{\gamma(\phi)
R}(f\mid\phi)R_{\mathcal{A}(\phi)}$ match on $\gamma B_\phi.$ But,
since $\gamma(\phi)>0,$ the measures on $\mathcal{A}(\phi)$ are characterized by their values on $\gamma B_\phi.$ Consequently, $E_R(f\mid\phi)R_{\mathcal{A}(\phi)}=E_{\gamma(\phi)
R}(f\mid\phi)R_{\mathcal{A}(\phi)}.$ This completes the proof of statement (1).
 \Boulette{(b)} It is a particular instance of statement (c), taking
 $Q=0.$

    \Boulette{(c)}
If $R$ and $Q$ are bounded, it is sufficient to take $\gamma\equiv 1.$
Suppose now that $R+Q$ is unbounded. The intersection
of two partitions which are  respectively $\sigma$-finite with respect to $R_\phi$ and
$Q_\phi$ is a  partition  $\seq\XX n$ of $\XX$ which is simultaneously
$\sigma$-finite with respect to $R_\phi$ and $Q_\phi.$  We assume without loss of generality that  $(R_\phi+Q_\phi)(\XX_n)\ge1$ for all $n.$
Let us define
$$
\gamma:=\sum_{n\ge1}\frac{2^{-n}}{(R_\phi+Q_\phi)(\XX_n)}\1_{\XX_n}.
$$
It is a measurable function on $\XX.$ As $\IO
\gamma(\phi)\,d(R+Q)=1$ and $0<\gamma(\phi)\le1, (R+Q)\ae,$
$\gamma(\phi) (R+Q)$ is a probability measure that is equivalent to 
$R+Q$ and $L^1(R+Q)\subset L^1(\gamma(\phi) (R+Q)).$
\end{proof}

\begin{definition}[Extension of the conditional expectation]\label{def-LZ07}
With Lemma \ref{res-LZ08}, we see that $E_{\gamma(\phi) R}(\cdot\mid\phi)$
is an extension of $E_{R}(\cdot\mid\phi)$ from $L^1(R)$ to $L^1(\gamma(\phi) R)$.
 We denote
\begin{equation*}
    E_{R}(f\mid\phi):=E_{\gamma(\phi) R}(f\mid\phi),\quad  f\in L^1(\gamma(\phi) R)
\end{equation*}
where $\gamma$ is a function the existence of which is ensured by
Lemma \ref{res-LZ08}.
\end{definition}

\begin{theorem}\label{res-LZ05}
Let $R,Q\in\MO$ and $\phi:\Omega\to\XX$ a measurable mapping in the Polish space $\XX$. We suppose
that $Q\prec R,$ and also that  $R_\phi$ are $Q_\phi$ $\sigma$-finite measures on $\XX.$ Then,
\begin{enumerate}[(a)]
    \item $E_R(\cdot\mid\phi)$ and $E_Q(\cdot\mid\phi)$ admit respectively a regular conditional probability kernel $x\in\XX\mapsto R(\cdot\mid\phi=x)\in\PO$
and $x\in\XX\mapsto Q(\cdot\mid\phi=x)\in\PO.$
    \item $\displaystyle{Q_\phi\prec R_\phi,\quad \frac{dQ}{dR}\in L^1(\gamma(\phi)
    R)}$
    and $$
\frac{dQ_\phi}{dR_\phi}(x)=E_R\left(\frac{dQ}{dR}\mid\phi=x\right),\quad
\forall x\in\XX,\ R_\phi\ae
    $$
    The function $\gamma$ in the above formulas is the one whose existence is ensured by Lemma  \ref{res-LZ08}-(c); it also appears in  Definition \ref{def-LZ07}.
     \item Moreover, $Q(\cdot\mid\phi)\prec R(\cdot\mid\phi),$
     $Q\ae$ and
\begin{equation}\label{eq-LZ53}
    \frac{dQ}{dR}(\omega)=\frac{dQ_\phi}{dR_\phi}(\phi(\omega))\frac{dQ(\cdot\mid\phi
       =\phi(\omega))}{dR(\cdot\mid\phi=\phi(\omega))}(\omega),
       \quad\forall \omega\in\Omega,\ Q\ae
\end{equation}

        \item A formula, more practical  than \eqref{eq-LZ53} is the following one. For any bounded measurable function $f$, we have
\begin{equation}\label{eq-LZ52}
    E_Q(f\mid\phi)=\frac{E_R\left(\frac{dQ}{dR}
    f\mid\phi\right)}{E_R\left(\frac{dQ}{dR}\mid\phi\right)},\quad Q\ae
\end{equation}
where no division by zero occurs since
$E_R\left(\frac{dQ}{dR}\mid\phi\right)>0, Q\ae$
\end{enumerate}
\end{theorem}

Identity \eqref{eq-LZ53} also writes more synthetically as
\begin{equation*}
    \frac{dQ}{dR}(\omega)=\frac{dQ_\phi}{dR_\phi}(\phi(\omega))\frac{dQ(\cdot\mid\phi)}{dR(\cdot\mid\phi)}(\omega),
     \quad \forall \omega\in\Omega,\ Q\ae
\end{equation*}
or more enigmatically as
\begin{equation*}
    \frac{dQ}{dR}(\omega)
   =\frac{dQ_\phi}{dR_\phi}(x)\frac{dQ(\cdot\mid \phi=x)}{dR(\cdot\mid
   \phi=x)}(\omega),\quad \forall(\omega,x),\ Q_\phi(dx)R(d\omega\mid
   \phi=x)\ae
\end{equation*}
since we have $\phi(\omega)=x,$ $Q_\phi(dx)R(d\omega\mid
\phi=x)$-almost surely.

\begin{proof}[Proof of Theorem \ref{res-LZ05}]  If $R$ and $Q$ are bounded measures, this theorem is an immediate consequence of Propositions \ref{res-LZ06} and
\ref{res-LZ04}.
\\
When $R_\phi$ and $Q_\phi$ are
$\sigma$-finite, we are allowed to invoke Lemma \ref{res-LZ08}:
$\gamma(\phi)R$ and $\gamma(\phi)Q$ are bounded measures and we can apply (i) to them. But, 
\begin{equation*}
  \frac{dQ}{dR} =  \frac{d(\gamma(\phi)Q)}{d(\gamma(\phi)R)}\qquad \textrm{and}\qquad
  \frac{dQ_\phi}{dR_\phi} = \frac{d(\gamma Q_\phi)}{d(\gamma R_\phi)}.
\end{equation*}
This completes the proof of the theorem.
\end{proof}

\subsection*{Hilbertian conditional expectation}

So far, we have considered the conditional expectation of a function $f$ in $L^1(R).$ If the reference measure $R$ is bounded, then $L^2(R)\subset L^1(R).$ But if $R$ is unbounded, this inclusion fails and  the   conditional expectation which we have just built is not valid for every $f$ in
$L^2(R).$ It is immediate to extend this notion from $L^1(R)\cap
L^2(R)$ to $L^2(R)$, interpreting the fundamental relation
\eqref{eq-LZ03} in restriction to $L^2(R)$:
\begin{equation*}
\IO hf\,dR= \IO h E_R(f\mid\mathcal{A})\,dR,\quad
     \forall h\in B_\mathcal{A}, f\in L^1(R)\cap L^2(R),
\end{equation*}
as an Hilbertian projection. We thus define the operator
$$
E_R(\cdot\mid \mathcal{A}):L^2(R)\to L^2(R_\mathcal{A})
$$
as an orthogonal projection on the Hilbertian subspace 
$L^2(R_\mathcal{A}).$ In particular, when $\mathcal{A}$ is the $\sigma$-field generated by the measurable mapping $\phi:\Omega\to\XX,$
\begin{equation*}
    E_R(\cdot\mid \phi):L^2(\Omega,R)\to L^2(\XX,R_\phi)
\end{equation*}
is specified for any function $f\in L^2(R)$ by
\begin{equation*}
 \IO h(\phi(\omega)) f(\omega)\,R(d\omega)=  \IX h(x) E_R(f\mid
 \phi=x)\,R_\phi(dx),
 \quad \forall h\in L^2(\XX,R_\phi).
\end{equation*}

\subsection*{Conditional expectation of path measures}

Now we particularize $\OO$ to be the path space $D(\ii,\XX)$ or $C(\ii,\XX).$

\begin{lemma}\label{res-01}
Let $Q\in\MO$ be a path measure and $\mathcal{T}\subset\ii$ a time subset. For $Q_{\mathcal{T}}$ to be a 
$\sigma$-finite measure, it is sufficient that there is some $t_o\in\mathcal{T}$ such that
$Q_{t_o}$ is a $\sigma$-finite measure.
\end{lemma}

\begin{proof}
 Let $t_o\in
\mathcal{T}$ be such that $Q_{t_o}\in\MX$ is a $\sigma$-finite measure with
$\seq{\XX}n$ an increasing sequence of measurable sets such that
$Q_{t_o}(\XX_n)<\infty$ and $\cup\XX_n=\XX.$ Then,
$Q_{\mathcal{T}}$ is also $\sigma$-finite, since
$Q_{\mathcal{T}}(X_{t_o}\in\XX_n)=Q_{t_o}(\XX_n)$ for all $n$ and
$\cup_{n\ge1}
[\Omega_\mathcal{T}\cap\{X_{t_o}\in\XX_n\}]=\Omega_\mathcal{T}.$
\end{proof}

\begin{definitions}[Conditionable path measure]\label{def-06}\
\begin{enumerate}
    \item A positive measure  $Q\in\MO$ is called a path measure.

    \item  The path measure $Q\in\MO$ is said to  be \emph{conditionable} if for all $t\in\ii,$  $Q_t$
is a $\sigma$-finite measure on $\XX.$
\end{enumerate}
\end{definitions}

 With Lemma \ref{res-01}, for any  conditionable path measure $Q\in\MO,$ the conditional expectation $E_Q(\cdot\mid X_\mathcal{T})$ is well-defined for any $\mathcal{T}\subset\ii.$ This is the reason for this definition.

Even when $Q(\OO)=\infty$,   Proposition
\ref{res-LZ06} tells us that $Q(\cdot\mid X_\mathcal{T})$ is a probability measure. In particular, $Q(B\mid X_\mathcal{T})$ and
$E_Q(b\mid X_\mathcal{T})$ are bounded measurable functions for any measurable subset $B$ and any measurable bounded function $b.$

\begin{example}
Let  $Q\in\MO$ the law of the
real-valued  process $X$ such that for all $0\le t<1$, $X_t=X_0$  is distributed with Lebesgue measure and $X_1=0,$ $Q$-almost  everywhere. We see with Lemma \ref{res-01} that
$Q=Q_{01}$ is a $\sigma$-finite measure since $Q_0$ is $\sigma$-finite. But $Q_1$ is not a 
$\sigma$-finite measure. Consequently, $Q$ is not a 
conditionable path measure.
\end{example}

\begin{definition}[Markov measure]
The path measure $Q\in\MO$ is said to be Markov if it is conditionable in the sense of Definition \ref{def-06} and if for all $0\le t\le1$ 
$$
Q(X _{[t,1]}\in\cdot\mid X _{[0,t]})=Q(X _{[t,1]}\in\cdot \mid X_t).
$$
\end{definition}

\section{Relative entropy with respect to an unbounded measure}\label{sec-entropy}

Let $R\in\MO$ be some $\sigma$-finite positive measure on some measurable  space $\OO$. The relative entropy of the probability measure $P\in\PO$ with respect to $R$ is loosely defined by
\begin{equation*}
H(P|R):=\IO \log(dP/dR)\, dP\in (-\infty,\infty],\qquad P\in \PO
\end{equation*}
if $P\ll R$ and $H(P|R)=\infty$ otherwise. 
\\
In the special case where $R$ is a probability measure, this definition is meaningful.

\begin{lemma}\label{res-RE03}
We assume that  $R\in\PO$ is a probability measure.
\\
We have for all $P\in\PO,$
$H(P|R)\in[0,\infty]$ and $H(P|R)=0$ if and only if $P=R.$
\\
The function $H(\cdot|R)$ is strictly convex on the convex set $\PO.$
\end{lemma}

\begin{proof}
We have $H(P|R)=\IO h\left(\frac{dP}{dR}\right)\,dR$ with $h(a)=a\log
a-a+1$ if $a>0$ and $h(0)=1.$ As $h\ge0,$  we see that for any $P\in\PO$ such that $P\ll R,$  $H(P|R)=\IO
h\left(\frac{dP}{dR}\right)\,dR\ge0.$ Hence
$H(P|R)\in[0,\infty].$  Moreover, $h(a)=0$ if and only if
$a=1.$ Therefore, $H(P|R)=0$ if and only if $P=R.$
\\
The strict convexity of $H(\cdot|R)$ follows from the strict convexity of $h.$
\end{proof}

If $R$ is unbounded, one must restrict the definition of $H(\cdot|R)$ to some subset of $\PO$ as follows. As $R$ is assumed to be $\sigma$-finite, there exists some  measurable function $W:\OO\to [0,\infty)$ such that
\begin{equation}\label{eq-Sc02}
z_W:=\IO e ^{-W}\, dR<\infty.
\end{equation}
Define the probability measure $R_W:= z_W ^{-1}e ^{-W}\,R$ so that $\log(dP/dR)=\log(dP/dR_W)-W-\log z_W.$ It follows that for any $P\in \PO$ satisfying $\IO W\, dP<\infty,$ the formula 
\begin{equation*}
H(P|R):=H(P|R_W)-\IO W\,dP-\log z_W\in (-\infty,\infty]
\end{equation*}
is a meaningful definition of the relative entropy which is coherent in the following sense. If $\IO W'\,dP<\infty$ for another measurable function $W':\OO\to[0,\infty)$ such that $z_{W'}<\infty,$ then $H(P|R_W)-\IO W\,dP-\log z_W=H(P|R _{W'})-\IO W'\,dP-\log z_{W'}\in (-\infty,\infty]$.
\\
Therefore, $H(P|R)$ is well-defined for any $P\in \PO$ such that $\IO W\,dp<\infty$ for some measurable nonnegative function $W$ verifying \eqref{eq-Sc02}. For any such function, let us define
$$
\mathrm{P}_W(\OO):=\left\{P\in\PO; \IO W\,dP<\infty\right\}.
$$ and $B_W(\OO)$ the space of measurable functions $u:\OO\to\RR$ such that $\sup_\OO |u|/(1+W)<\infty.$
When $\OO$ is a topological space, we also define the space $C_W(\OO)$ of all  continuous functions on $\OO$ such that $\sup_\OO |u|/(1+W)<\infty$.

\begin{proposition}\label{resL-13}
Let $W$ be some function which satisfies \eqref{eq-Sc02}.
For all $P\in \mathrm{P}_W(\OO)$,
\begin{equation}\label{eqAp-c}
\begin{split}
  H(P|R)&= \sup\left\{\int u\,dP-\log\int e^u\,dR; u\in B_W(\OO)\right\}\\
        &= \sup\left\{\int u\,dP-\log\int e^u\,dR; u\in C_W(\OO)\right\}
\end{split}
\end{equation}
and for all $P\in\PO$ \textrm{ such that } $P\ll R,$
\begin{equation}\label{eqAp-d}
    H(P|R) =\sup\left\{\int u\,dP-\log\int e^u\,dR; u:  \int
    e^{u}\,dR<\infty, \int u_-\,dP<\infty
    \right\}
\end{equation}
where $u_-=(-u)\vee 0$ and $\int u\,dP\in(-\infty,\infty]$ is
well-defined for all $u$ such that $\int u_-\,dP<\infty.$
\end{proposition}

In \eqref{eqAp-c}, when $C_W(\OO)$ is invoked, it implicitly assumed that $\OO$ is a topological space equipped with its Borel $\sigma$-field.

The proof below is mainly a rewriting of the proof of \cite[Prop.\,B.1]{GL10} in the setting where the reference measure is possibly unbounded.

\begin{proof}[Proof of Proposition \ref{resL-13}] Once we
have \eqref{eqAp-d}, \eqref{eqAp-c} follows by standard
approximation arguments.
\\
The proof of \eqref{eqAp-d} relies on Fenchel inequality for the
convex function $h(t)=t\log t$: 
$$
st\le t\log t+e ^{s-1}
$$ 
for all $s\in[-\infty,\infty),$ $t\in[0,\infty),$
with the conventions $0\log 0=0,$ $e^{-\infty}=0$ and
$-\infty\times 0=0$ which are legitimated by limiting procedures.
The equality is attained when $t=e ^{s-1}.$
\\
Taking $s=u(x),$  $t=\frac{dP}{dR}(x)$ and integrating with
respect to $R$ leads us to 
$$
\int u\,dP\le H(P|R)+\int e ^{u-1}\,dR,
$$ 
whose terms are meaningful with
values in $(-\infty,\infty],$ provided that $\int
u_-\,dP<\infty$ and $\IO e^u\, dR<\infty.$ Formally, the case of equality corresponds to
$\frac{dP}{dR}=e ^{u-1}.$ With the monotone convergence theorem,
one sees that it is approached by the sequence $u_n=1+
\log(\frac{dP}{dR}\vee e^{-n}),$ as $n$ tends to infinity.
This gives us 
$$H(P|R)=\sup\left\{\int u\,dP-\int
e ^{u-1}\,dR; u: \int e^{u}\,dR<\infty,\inf u>-\infty
\right\},$$ 
which in turn implies that
\begin{equation*}
    H(P|R)=\sup\left\{\int u\,dP-\int e ^{u-1}\,dR; u: \int
e^{u}\,dR<\infty,\int u_-\,dP<\infty \right\}.
\end{equation*}
Now, we take advantage of the unit mass of $P\in\PO:$ 
$$\int
(u+b)\,dP-\int e^{u+b-1}\,dR=\int u\,dP-e ^{b-1}\int
e^u\,dR +b,\quad \forall b\in \RR,
$$ 
and we use the easy
identity $\log a=\inf_{b\in\RR}\{ae ^{b-1}-b\}$ to obtain
    $$\sup_{b\in\RR}\left\{\int (u+b)\,dP-\int
e^{u+b-1}\,dR\right\}=\int u\,dP-\log\int e^u\,dR.$$
Whence,
\begin{eqnarray*}
  &&\sup\left\{\int u\,dP-\int e^{u-1}\,dR; u: \int e^{u}\,dR<\infty,\int u_-\,dP<\infty
  \right\}\\
    &=& \sup\left\{\int (u+b)\,dP-\int e^{u+b-1}\,dR;b\in\RR,  u: \int e^{u}\,dR<\infty,\int u_-\,dP<\infty \right\}\\
  &=& \sup\left\{\int u\,dP-\log\int e^u\,dR; u: \int e^{u}\,dR<\infty,\int u_-\,dP<\infty
  \right\}.
\end{eqnarray*}
This completes the proof of \eqref{eqAp-d}.
\end{proof}

Let $W$ be a nonnegative measurable function on $\OO$ that verifies \eqref{eq-Sc02}.
Let us introduce the space $\mathrm{M}_W(\OO)$ of all signed measures $Q$ on $\OO$ such that $\IO W\,d|Q|<\infty.$

\begin{corollary}\label{res-10}

The function $H(\cdot|R)$ is convex on the vector space of all  signed measures. Its effective domain $\dom H(\cdot|R):=\left\{H(\cdot|R)<\infty\right\} $ is included in $\mathrm{P}_W(R)$

Suppose furthermore that $\OO$ is a topological space. Then, $H(\cdot|R)$ is \lsc\ with respect to the topology $\sigma(\mathrm{M}_W(\OO),C_W(\OO)).$

As a function of its two arguments  on $\mathrm{M}_W(\OO)\ttimes\mathrm{M}_W(\OO),$
$H(\cdot\mid \cdot)$ is jointly convex and jointly \lsc\ with respect to the product topology. In particular, it is a jointly Borel function.
\end{corollary}

\begin{proof}
The first statement follows from \eqref{eqAp-c}.
\\
With Proposition \ref{resL-13}, we see that $H(\cdot|R)$ is the supremum of a family of affine continuous functions: $Q\mapsto \IO u\,dQ-\log\IO e^u\,dR$ indexed by $u.$
Hence, it is convex and \lsc. The same argument works with the joint arguments.
\end{proof}

Let $\OO$ and $Z$ be two  Polish spaces equipped with their Borel $\sigma$-fields. For any measurable function $\phi:\OO\to Z$ and any measure $Q\in \MO$ we have the  disintegration formula
\begin{equation*}
Q(\cdot)=\int _{Z} Q(\cdot\mid\phi=z)\, Q_\phi(dz)
\end{equation*}
where we write $Q_\phi:=\phi\pf Q$ and  $z\in Z\mapsto Q(\cdot|\phi=z)\in \PO$ is measurable.

\begin{theorem}[Additive property  of the relative entropy]\label{res-02}
We have
$$
H(P|R)=H(P_\phi|R_\phi)+\int_Z
H\Big(P(\cdot\mid\phi=z)\Big|R(\cdot\mid\phi=z)\Big)\,P_\phi(dz),\quad P\in\PO.
$$
\end{theorem}

\begin{proof}
By Theorem \ref{res-LZ05}, 
\begin{eqnarray*}
&&H(P|R)
 = \int_Z E_P\left[\log(\frac{dP}{dR})\mid\phi=z\right]\,P_\phi(dz) \\
  &=& \int_Z
  \log\frac{dP_\phi}{dR_\phi}(z)\,P_\phi(dz)
  +\int_Z \left[\int_\Omega\log\frac{dP(\cdot\mid\phi=z)}{dR(\cdot\mid\phi=z)}(\omega)\,P(d\omega\mid\phi=z)\right]\,P_\phi(dz)
\end{eqnarray*}
which is the announced result.
\end{proof}
\begin{remarks} There are serious measurability problems hidden behind this proof.

\begin{enumerate}[(a)]
    \item The assumption that $Z$ is Polish ensures the existence of kernels $z\mapsto P(\cdot\mid\phi=z)$ and $z\mapsto
    R(\cdot\mid\phi=z).$ On the other hand, we know that for any function $u\in
    B_W,$ the mapping $z\in\XX\mapsto E_P(u\mid\phi=z)\in\RR$ is
    measurable. Therefore, the mapping $z\in Z\mapsto
    P(\cdot\mid\phi=z)\in \mathrm{P}_W(\OO)$ is measurable once $\mathrm{P}_W(\OO)$ is equipped with its cylindrical $\sigma$-field, i.e.\,generated by the mappings $Q\in \mathrm{P}_W(\OO)\mapsto \IO
    u\,dQ$ where $u$ describes $B_W.$ But this $\sigma$-field matches with the Borel $\sigma$-field of
    $\sigma(\mathrm{P}_W(\OO),C_W)$ when $\Omega$ is metric and separable.
    As $H$ is jointly Borel (see Corollary \ref{res-10}), it is jointly measurable with respect to the product of the cylindrical $\sigma$-fields. Hence,  $z\mapsto
    H\Big(P(\cdot\mid\phi=z)\Big|R(\cdot\mid\phi=z)\Big)$ is
    measurable.
    \\ Note that in general, the Borel $\sigma$-field of
    $\sigma(\mathrm{P}_W(\OO),B_W)$  is too rich to match with the cylindrical $\sigma$-field. This is the reason why 
    $\Omega$ is assumed to be Polish (completeness doesn't play any role here).

    \item The relative entropy
$H\Big(P(\cdot\mid\phi=z)\Big|R(\cdot\mid\phi=z)\Big)$ inside the second integral of the additive  property formula is a function of couples of probability measures. Therefore, with 
Lemma \ref{res-RE03}, we know that it is nonnegative in general and that it vanishes if and only if $P(\cdot\mid\phi=z)=R(\cdot\mid\phi=z).$
    \item
    Together with its measurability, which was proved at Remak (a) above, this allows us to give a meaning to the integral  $\int_Z
H\Big(P(\cdot\mid\phi=z)\Big|R(\cdot\mid\phi=z)\Big)\,P_\phi(dz)$
in $[0,\infty].$
\end{enumerate}
\end{remarks}

Let us mention an application of this theorem in the context of the Schr\"odinger problem \eqref{eq-03} where $\OO$ is a path space, see \cite{Foe85,Leo12e}. For any, $R\in\MO,$ $P\in\PO,$ we have
\begin{equation*}
H(P|R)=H(P _{01}|R _{01})+\IXX H(P ^{xy}|R ^{xy})\,P _{01}(dxdy)
\end{equation*}
where $Q _{01}:= (X_0,X_1)\pf Q$ is the law of the endpoint position  and $Q ^{xy}:=Q(\cdot|X_0=x,X_1=y)$ is the bridge from $x$ to $y$  under  $Q$. From this additive  property formula and Corollary \ref{res-10}, it is easily seen that the solution $\Ph$ of \eqref{eq-03} (it is unique, since the entropy is \emph{strictly} convex) satisfies
\begin{equation*}
\Ph ^{xy}=R ^{xy},\quad \forall (x,y)\in\XXX, \Ph _{01}\ae
\end{equation*}
and that $\Ph _{01}$ is the unique solution of
\begin{equation*}
H(\pi|R _{01})\to \textrm{min};\qquad \pi\in\PXX:\pi_0=\mu_0, \pi_1=\mu_1
\end{equation*}
where $\pi_0$ and $\pi_1\in\PX$ are the first and second marginals of $\pi\in\PXX.$

\section{Positive integration with respect to a Markov measure}\label{sec-positive}

\subsection*{Integration of nonnegative functions}

The expectation  $E_R Z$ of a \emph{nonnegative} random variable $Z$
with respect to a positive $\sigma$-finite measure $R$ is a well-defined notion, even when $Z$ is not $R$-integrable; in which case, one sets  $E_R Z=+\infty.$ Indeed, with the monotone convergence theorem we have
\begin{equation*}
    E_R Z=\Lim n E_R[\1_{\{\cup_{k\le n}\Omega_k\}}(Z\wedge n)]\in
[0,\infty]
\end{equation*}
where $\seq \Omega k$ is  a $\sigma$-finite partition of  $R.$

Since $R(\cdot\mid \mathcal{A})$ is a bounded measure, we see that
$E_R(Z\mid \mathcal{A})$ is well defined in $[0,\infty].$ Moreover, the fundamental formula of the the conditional expectation is kept:
\begin{equation*}
    E_R[aE_R(Z\mid \mathcal{A})]=E_R(aZ)
\end{equation*}
for any nonnegative  function $a\in \mathcal{A}.$ To see this, denote
$a_n=\1_{\{\cup_{k\le n}\Omega_k\}}(a\wedge n)$ and
$Z_n=\1_{\{\cup_{k\le n}\Omega_k\}}(Z\wedge n).$ We have
$E_R[a_nE_R(Z_n\mid \mathcal{A})]=E_R(a_nZ_n)$ for all $n\ge1.$ Letting  $n$ tend to infinity, we obtain the announced identity with the monotone convergence theorem.

\subsection*{Positive integration with respect to a Markov measure}

We present a technical lemma about positive integration with respect to a Markov measure $R\in\MO.$ It is an easy result, but it is rather practical. It allows to work with $(f,g)$-transforms of Markov processes without assuming unnecessary integrability conditions on $f$ and $g$.

\begin{lemma}\label{res-24} Let $R\in\MO$ be a Markov measure.
\begin{enumerate}[(a)]
    \item Let  $0\le t\le1$ and $\alpha ,\beta $ be nonnegative functions such that $\alpha \in\AAA 0t$ and 
$\beta \in\AAA t1.$ Then, for any $\omega$ outside an $R$-negligible set:
\begin{enumerate}[(i)]
    \item if $E_R(\alpha \beta \mid X_t)(\omega)=0,$ we have $E_R(\alpha \mid X_t)(\omega)=0$ or $E_R(\beta \mid X_t)(\omega)=0;$
    \item if $E_R(\alpha \beta \mid X_t)(\omega)>0,$ we have  $E_R(\alpha \mid
X_t)(\omega),E_R(\beta \mid X_t)(\omega)>0$ and
    $
    E_R(\alpha \beta \mid X_t)(\omega)=E_R(\alpha \mid X_t)(\omega)E_R(\beta \mid
    X_t)(\omega)\in (0,\infty].
    $
\end{enumerate}
    \item Let $P\in\MO$ be a conditionable path measure such that $P\prec R$ and whose density writes as $\displaystyle{\frac{dP}{dR}=\alpha
\beta }$ with $\alpha,\beta$  nonnegative functions such that $\alpha \in\AAA 0t$ and $\beta \in\AAA t1$ for some
$0\le t\le1.$ Then,
$$
\left\{
\begin{array}{l}
  E_R(\alpha \mid X_t),E_R(\beta \mid X_t)\in (0,\infty) \\
  E_R(\alpha \beta \mid X_t)=E_R(\alpha \mid X_t)E_R(\beta \mid
    X_t)\in (0,\infty)\\
\end{array}
    \right.\quad P\ae
$$
(but not  $R\ae$ in general). Furthermore,
\begin{eqnarray}\label{eq-33a}
    &&E_R(\alpha \beta \mid X_t)\\\nonumber
    =&&\1_{\{E_R(\alpha \mid X_t)<\infty, E_R(\beta \mid X_t)<\infty\}}
    E_R(\alpha \mid X_t)E_R(\beta \mid X_t)\in [0,\infty)\quad R\ae
\end{eqnarray}
\end{enumerate}
 \end{lemma}

As regards \eqref{eq-33a}, even if $\alpha
\beta $ is integrable, it is not true in general that the nonnegative functions $\alpha $ and $\beta $ are integrable. Therefore, a priori the conditional expectations  $E_R(\alpha \mid
X_t)$ and $E_R(\beta \mid X_t)$ may be infinite.

\begin{proof}
\boulette{(a)}  The  measure $R$  disintegrates with respect to the initial and final positions:
\begin{equation*}
    R=\IX R(\cdot\mid X_0=x)\,R_0(dx)=\IX R(\cdot\mid
    X_1=y)\,R_1(dy)
\end{equation*}
But,   $R_0$ and $R_1$ are assumed to be $\sigma$-finite measures. Let $\seq{\XX^0}n$ and $\seq{\XX^1}n$ be two $\sigma$-finite partitions of $R_0$ and $R_1,$ respectively. We denote
$\Omega^0_n=\{X_0\in \cup _{k\le n}\XX^0_k\},$ $\Omega^1_n=\{X_1\in\cup _{k\le n} \XX^1_k\}$ and
$\Omega_n=\Omega_n^0 \cap \Omega_n^1.$

As $R$ is Markov, if the functions $\alpha $ and
$\beta $ are integrable, then $E_R(\alpha \mid X_t)$ are
$E_R(\beta \mid X_t)$ well-defined and
\begin{equation*}
    E_R(\alpha \beta
\mid X_t)=E_R(\alpha \mid X_t)E_R(\beta \mid X_t).
\end{equation*}
Letting $n$ tend to infinity in $E_R[(\alpha \wedge n)(\beta \wedge
n)\1_{\Omega_n}\mid X_t]=E_R((\alpha \wedge
n)\1_{\Omega_n^0}\mid X_t)E_R((\beta \wedge
n)\1_{\Omega_n^1}\mid X_t),$ we obtain
$E_R(\alpha \beta \mid X_t)=E_R(\alpha \mid X_t)E_R(\beta \mid
X_t)\in[0,\infty].$  One concludes, remarking that the sequences are increasing.

\Boulette{(b)} It is a consequence of the first part of the lemma. As $P_t$ is  $\sigma$-finite measure,
$\frac{dP_t}{dR_t}(X_t)<\infty,$ $R\ae$ (hence, a fortiori  $P\ae).$
In addition, $\frac{dP_t}{dR_t}(X_t)>0,$ $P\ae$ (but not $R\ae$ in general) and $\frac{dP_t}{dR_t}(X_t)=E_R(\alpha \beta \mid X_t)$, by 
Theorem \ref{res-LZ05}-(b). Consequently, we are allowed to apply part (ii) of (a) to obtain the identity which holds $P\ae$ This identity extends $R\ae$, yielding
\eqref{eq-33a}. To see this, remark with part (i) of (a) that when the density vanishes, the two terms of the product cannot be \emph{simultaneously} equal to $\infty$ and one of them vanishes.
 \end{proof}

Analogously, one can prove the following extension.

\begin{lemma}\label{res-24b} Let $R\in\MO$ be a Markov measure.
\begin{enumerate}
    \item Let  $0\le s\le t\le1$ and two nonnegative functions $\alpha ,\beta $ such that  $\alpha \in\AAA 0s,$
$\beta \in\AAA t1.$ Then, for any $\omega$ outside an $R$-negligible set:
\begin{enumerate}[(a)]
    \item if $E_R(\alpha \beta \mid X_{[s,t]})(\omega)=0,$ we have $E_R(\alpha \mid X_s)(\omega)=0$ or $E_R(\beta \mid X_t)(\omega)=0;$
    \item if $E_R(\alpha \beta \mid X_{[s,t]})(\omega)>0,$ we have $E_R(\alpha \mid
X_s)(\omega),E_R(\beta \mid X_t)(\omega)>0$ and
    $
    E_R(\alpha \beta \mid X_{[s,t]})(\omega)=E_R(\alpha \mid X_s)(\omega)E_R(\beta \mid
    X_t)(\omega)\in (0,\infty].
    $
\end{enumerate}
    \item Let $P\in\MO$ be a conditionable path measure such that $P\prec R$ and whose density writes as
    $\displaystyle{\frac{dP}{dR}=\alpha\zeta
\beta }$ with $\alpha, \zeta$ and $\beta$ nonnegative functions such that $\alpha \in\AAA 0s, \zeta\in\AAA st$ and $\beta \in\AAA
t1$ for some $0\le s\le t\le1.$ Then,
$$
\left\{
\begin{array}{l}
  E_R(\alpha \mid X_s),E_R(\beta \mid X_t)\in (0,\infty) \\
  E_R(\alpha \beta \mid X_{[s,t]})=E_R(\alpha \mid X_s)E_R(\beta \mid
    X_t)\in (0,\infty)\\
\end{array}
    \right.\quad P\ae
$$
(and not  $R\ae$ in general). In addition,
\begin{eqnarray*} 
    &&E_R(\alpha\zeta \beta \mid X_{[s,t]})\\\nonumber
    =&&\1_{\{E_R(\alpha \mid X_s)<\infty, E_R(\beta \mid X_t)<\infty\}}
    E_R(\alpha \mid X_s)\zeta E_R(\beta \mid X_t)\in [0,\infty)\quad R\ae
\end{eqnarray*}
\end{enumerate}
 \end{lemma}

\appendix

\section{Conditional expectation with respect to an unbounded measure}\label{sec-CE}

In standard textbooks, the theory of conditional expectation  is presented and developed with respect to a probability measure (or equivalently, a bounded positive measure). However, there are natural unbounded path measures, such as the reversible Brownian motion on $\Rn,$ with respect to which a conditional expectation theory is needed. We present the details of this notion in this appendix section. From a measure theoretic viewpoint, this section is about disintegration of unbounded positive measures.

\subsection*{The role of $\sigma$-finiteness in Radon-Nikodym theorem}

The keystone of conditioning is Radon-Nikodym theorem. In order to emphasize the role of $\sigma$-finiteness, we recall a classical proof of this theorem, following von Neumann and Rudin, \cite{Rud-Analysis}. Let $\OO$ be a space with its $\sigma$-field and $P,Q,R\in\MO$ be positive measures on $\OO.$ One says that $P$ is absolutely continuous with respect to $R$ and denotes $P\ll R,$ if for every measurable subset $A\subset \OO,$ $R(A)=0\Rightarrow P(A)=0.$ It is said to be concentrated on the measurable subset $C\subset\OO$ if for any measurable subset $A\subset\OO,$ $P(A)=P(A\cap C).$ The measures $P$ and $Q$ are said to be mutually singular and one denotes $P\bot Q,$ if there exist two disjoint measurable subsets $C,D\subset\OO$ such that $P$ is concentrated on $C$ and $Q$ is concentrated on $D.$

\begin{theorem}\label{res-LZ32}
Let $P$ and $R$ be two bounded positive measures.
\begin{enumerate}[(a)]
    \item There exists a unique pair $(P_a,P_s)$  of measures
    such that $P=P_a+P_s,$ $P_a\ll R$ and $P_s\bot R.$
    These measures are positive and  $P_a\bot P_s.$
    \item There is a unique function $\theta \in L^1(R)$ such that
    \begin{equation*}
    P_a(A)=\int_A \theta \, dR, \quad \textrm{for any measurable subset }A.
\end{equation*}
\end{enumerate}
\end{theorem}

\begin{proof}
The uniqueness proofs are easy. Let us begin with (a). Suppose we have two Lebesgue decompositions: $P=P_a+P_s=P_a'+P_s'.$
Then, $P_a-P_a'=P_s'-P_s,$ $P_a-P_a'\ll R$ and $P_s'-P_s\bot R.$
Hence, $P_a-P_a'=P_s'-P_s=0$ since $Q\ll R$ and $Q\bot R$ imply that
 $Q=0.$ As regards (b), if we have $P_a=\theta
R=\theta 'R,$ then $\int_A (\theta -\theta ')\, dR=0$ for any measurable $A\subset\OO.$ Therefore $\theta =\theta ', R\ae$

Denote $Q=P+R.$ It is a bounded positive measure and for any function $f\in L^2(Q),$
\begin{equation}\label{eq-LZ38}
    |\IO f\,dP|\le \IO |f|\, dQ\le
\sqrt{Q(\Omega)} \|f\|_{L^2(Q)}.
\end{equation}
It follows that $f\in L^2(Q)\mapsto \IO f\,dP\in\mathbb{R}$ is a continuous linear form on the  Hilbert space $L^2(Q).$ Consequently, there exists $g\in L^2(Q)$ such that
\begin{equation}\label{eq-LZ36}
    \IO f\,dP=\IO fg\,dQ,\quad \forall f\in L^2(Q).
\end{equation}
Since $0\le P\le P+R:=Q,$ we obtain $0\le g\le 1,$
$Q\ae$ Let us take a version of $g$ such that $0\le g\le 1$ everywhere.
The identity \eqref{eq-LZ36} rewrites as
\begin{equation}\label{eq-LZ37}
    \IO (1-g)f\,dP=\IO fg\,dR,\quad \forall f\in L^2(Q).
\end{equation}
Let us set  $C:=\{0\le g<1\},$ $D=\{g=1\},$ $P_a(\cdot)=P(\cdot\cap C)$
et $P_s(\cdot)=P(\cdot\cap D).$
\\
Choosing $f=\1_D$ in \eqref{eq-LZ37}, we obtain $R(D)=0$ so that $P_s\bot R.$
\\
Choosing $f=(1+g+\cdots+g^n)\1_A$ with $n\ge1$ and $A$ any measurable subset in \eqref{eq-LZ37}, we obtain
\begin{equation*}
    \int_A (1-g^{n+1})\,dP=\int_A g(1+g+\cdots+g^n)\,dR.
\end{equation*}
But the sequence of functions $(1-g^{n+1})$ increases  pointwise towards $\1_C.$ Now, by the monotone convergence theorem, we have $P(A\cap C)=\int_A \1_C g/(1-g)\,dR.$ This means that
$P_a=\theta R$ with $\theta =\1_{\{0\le g<1\}} g/(1-g).$
\\
Finally, we see that $\theta \ge0$ is $R$-intégrable since
$\IO \theta \,dR=P_a(\Omega)\le P(\Omega)<\infty.$
\end{proof}

The main argument of this proof is Riesz theorem on the representation of the dual of a Hilbert space. As the continuity of the linear form is ensured by $Q(\OO)<\infty$ at \eqref{eq-LZ38}, we have used crucially the boundedness of the measures $P$ and $R.$ This can be relaxed by means of the following notion.

\begin{definition}\label{def-LZ01}
The positive measure $R$ is said to be $\sigma$-finite if it is either bounded or if there exists  a sequence
$\seq{\Omega}k$ of disjoint measurable subsets which partitions 
$\Omega:$ $\sqcup_k\Omega_k=\Omega$ and are such that
$R(\Omega_k)<\infty$ for all $k.$
\\
In such a case it is said that $\seq{\Omega}k$ finitely partitions  $R$ or that it is a  $\sigma$-finite partition of $R.$
\end{definition}

Recall that an unbounded positive measure is allowed to take the value 
$+\infty.$ For instance, the measure $R$ which is defined on the trivial $\sigma$-field $\{\emptyset,\Omega\}$ by $R(\emptyset)=0$ and
$R(\Omega)=\infty$ is a genuine positive measure and
$L^1(R)=\{0\}$. This situation may seem artificial, but in fact it is not, as can be observed with the following examples.

\begin{examples}\
\begin{enumerate}[(a)]
    \item The push-forward of Lebesgue measure on $\mathbb{R}$ by a function which  takes finitely many values is a positive measure on the set of these values which charges at least one of them with an infinite mass. Remark in passing that this provides us with an example of a  $\sigma$-finite measure whose pushed forward is not.
    \item Lebesgue measure on $\RR^2$ is $\sigma$-finite,
    but its push-forward by the projection on the first coordinate
    assigns an infinite mass to any non-negligible Borel set.
\end{enumerate}
\end{examples}

\begin{theorem}[Radon-Nikodym]\label{res-LZ33} Let $P$ and $R$ two positive $\sigma$-finite measures such that $P\ll R.$ Then, there exists a unique measurable function $\theta $ such that
\begin{equation}\label{eq-LZ35}
    \IO f\,dP=\IO f\theta \,dR,\quad \forall f\in L^1(P).
\end{equation}
Moreover, $P$ is bounded if and only if $\theta \in L^1(R).$
\end{theorem}

\begin{proof}
Taking the intersection of two partitions which respectively finitely partition $R$ and $P$, one obtains a countable measurable partition which simultaneously finitely partitions $R$ and $P.$ Theorem \ref{res-LZ32} applies on each subset of this partition and one obtains the desired result by recollecting the pieces. The resulting function $\theta$ need not be integrable anymore, but it is still is locally integrable in the sense that it is integrable in restriction to each subset of the partition. We have just extended Theorem \ref{res-LZ32} when the measures $P$ and $R$ are $\sigma$-finite. We conclude noticing that by Theorem \ref{res-LZ32} we have: $P\ll R$ if and only if $P_s=0.$
\end{proof}

As regards Radon-Nikodym theorem, making a step away from $\sigma$-finiteness seems to be hopeless, as one can guess from the following example. Take $R=\sum_{x\in
[0,1]}\delta_x:$ the counting measure on $\Omega=[0,1],$ and $P$
the Lebesgue measure on $[0,1].$ We see that $P\ll R,$ but there is no measurable function $\theta$ which satisfies \eqref{eq-LZ35}.

\subsection*{Conditional expectation with respect to a positive measure}

Let $\OO$ be a space furnished with some $\sigma$-field and  a sub-$\sigma$-field $\mathcal{A}$. We take a positive measure $R\in\MO$ on $\OO$ and denote $R _{\mathcal{A}}$ its restriction to $\mathcal{A}$. The space of bounded measurable functions is denoted by $B$, while $B_\mathcal{A}$ is the subspace of bounded $\mathcal{A}$-measurable functions. The subspace of $L^1(R)$ consisting of the $\mathcal{A}$-measurable integrable functions is denoted by $L^1(R_\mathcal{A})$ .

We take $g\ge0$ in $L^1(R).$ The mapping
$h\in B_\mathcal{A} \mapsto \IO hg\,dR:=\IO h\,dR^g_\mathcal{A}$ defines a \emph{finite} positive measure $R^g_\mathcal{A}$ on $(\Omega,\mathcal{A}).$
Clearly, if $h\ge0$ and $\IO h\,dR=\IO h\,dR_\mathcal{A}=0,$ then $\IO
h\,dR^g_\mathcal{A}=0.$ This means that $R^g_\mathcal{A}$ is a finite measure which is absolutely continuous with respect to $R_\mathcal{A}.$ If 
$R_\mathcal{A}$ is assumed to be $\sigma$-finite, by the Radon-Nikodym Theorem
\ref{res-LZ33}, there is a unique function $\theta_g\in L^1(R_\mathcal{A})$ such that
$R^g_\mathcal{A}=\theta_g R_\mathcal{A}.$ We have just obtained
    $\IO hg\,dR= \IO h \theta_g\,dR_\mathcal{A}=\IO h\theta_g\,dR,$
    $\forall h\in B_\mathcal{A}.$ Now, let $f\in
    L^1(R)$ which might not be nonnegative. Considering its decomposition $f=f_+-f_-$ into nonnegative and nonpositive parts:
    $f_+=f\vee 0,$ $f_-=(-f)\vee 0,$
    and setting $\theta_f=\theta_{f_+}-\theta_{f_-},$ we obtain
\begin{equation}\label{eq-LZ03}
    \IO hf\,dR= \IO h \theta_f\,dR_\mathcal{A}=\IO h\theta_f\,dR,\quad
     \forall h\in B_\mathcal{A}, f\in L^1(R).
\end{equation}

\begin{definition}[Conditional expectation]\label{def-LZ08} It is assumed that $R_\mathcal{A}$ is $\sigma$-finite. 
\\
For any $f\in L^1(R),$
the conditional expectation of $f$ with respect to  $\mathcal{A}$ is the unique (modulo $R\ae$-equality) function
$$E_R(f\mid\mathcal{A})\in L^1(R_\mathcal{A})$$
which is integrable, $\mathcal{A}$-measurable and such that
$\theta_f=:E_R(f\mid\mathcal{A})$ satisfies \eqref{eq-LZ03}.
\end{definition}
It is essential in this definition that $R_\mathcal{A}$ is assumed to be $\sigma$-finite.
\\
Of course,
\begin{equation}\label{eq-LZ06}
    E_R(f\mid\mathcal{A})=f,\quad \forall f\in L^1(R_\mathcal{A})
\end{equation}
If, in \eqref{eq-LZ03}, we take the function
$h=\mathrm{sign}(E_R(f|\mathcal{A}))$ which is in $B_\mathcal{A} ,$ we have
\begin{equation}\label{eq-LZ05}
    \IO |E_R(f\mid\mathcal{A})|\,dR_\mathcal{A}\le \IO |f|\,dR
\end{equation}
which expresses that $E_R(\cdot\mid\mathcal{A}): L^1(R)\to L^1(R_\mathcal{A})$ is a
contraction, the spaces $L^1$ being equipped with their usual norms
$\|\cdot\|_1$. With \eqref{eq-LZ06}, we see that the opertot norm of this contraction is 1. Therefore, $E_R(\cdot\mid\mathcal{A}): L^1(R)\to
L^1(R_\mathcal{A})$ is a continuous projection.
\\
Taking $h=1$ in \eqref{eq-LZ03}, we have
\begin{equation*}
    \IO f(\omega)\,R(d\omega)=\IO E_R(f\mid\mathcal{A})(\eta)\,R(d\eta),
\end{equation*}
which can be  written
\begin{equation}\label{eq-LZ08}
    E_R E_R(f\mid\mathcal{A})=E_R(f),
\end{equation}
with the notation $E_R(f):=\IO f\,dR.$

\begin{remark}\label{rem-LZ01}
When $R$ is a bounded measure, the mapping
$E_R(\cdot\mid\mathcal{A})$ shares the following properties.
\begin{enumerate}[(a)]
    \item For all $f\in L^1(R)\ge0,$ $E_R(f\mid\mathcal{A})\ge0,$ $R_\mathcal{A}\ae$
    \item $E_R(1\mid\mathcal{A})=1,$ $R_\mathcal{A}\ae$
    \item For all $f,g\in L^1(R)$ and  $\lambda\in\RR,$ $E_R(f+\lambda g\mid\mathcal{A})=E_R(f\mid\mathcal{A})+\lambda E_R(g\mid\mathcal{A}),$ $R_\mathcal{A}\ae$
    \item For any sequence $\seq fn$ in $L^1(R)$ with $0\le f_n\le1,$
    which converges \emph{pointwise} to 0, we have: $\Lim n
    E_R(f_n\mid\mathcal{A})=0,$ $R_\mathcal{A}\ae$
\end{enumerate}
Except for the ``$R_\mathcal{A}\ae$", these properties  characterize the expectation with respect to a probability measure. They can easily be checked, using \eqref{eq-LZ03}, as follows.
\begin{enumerate}[(i)]
    \item
     For any $h\in B_\mathcal{A} \ge0$ and $f\in L^1(R)\ge0,$ \eqref{eq-LZ03} implies that $\IO h
    E_R(f\mid\mathcal{A})\,dR_\mathcal{A}\ge0,$ which in turns implies  (a).
    \item
     For any $h\in B_\mathcal{A} \ge0,$ \eqref{eq-LZ03} implies that $\IO h
    E_R(1\mid\mathcal{A})\,dR_\mathcal{A}=\IO h\,dR_\mathcal{A},$ whence (b).
    \item The linearity of $f\mapsto E_R(f\mid\mathcal{A})(\eta)$ comes from the linearity of $f\mapsto \IO hf\,dR$ for all
    $h\in B_\mathcal{A} .$ Indeed, for all $f,g\in L^1(R)$
    and  $\lambda\in\RR,$ we have $\IO h E_R(f+\lambda g\mid\mathcal{A})\,dR=\IO h [E_R(f\mid\mathcal{A})+\lambda
    E_R(g\mid\mathcal{A})]\,dR,$ which implies (c).
    \item For any $h\in B_\mathcal{A} ,$ Fatou's lemma,  \eqref{eq-LZ03} and the dominated convergence theorem lead us to $0\le \IO h\Lim n E_R(f_n\mid\mathcal{A})\,dR_\mathcal{A}\le\Liminf n  \IO h E_R(f_n\mid\mathcal{A})\,dR_\mathcal{A}=\Lim n\IO hf_n\,dR=0.$
    This proves (d).
\end{enumerate}
We used the boundedness of $R$ at items (ii) and
(iv), since in this case, bounded functions are integrable.
\end{remark}

One could hope  that for $R_\mathcal{A}\ae$ $\eta,$ there exists a probability kernel $\eta\mapsto R(\cdot\mid\mathcal{A})(\eta)$ which admits $E_R(\cdot\mid\mathcal{A})$ as its expectation. But negligible sets have to be taken into account. Indeed, the $R_\mathcal{A}$-negligible sets which invalidate these equalities depend on the function $f,g$, the real numbers $\lambda$ and the sequences $\seq fn$. Their non-countable union might not be measurable, and even in this case the measure of this  union might be positive. Therefore, the $\sigma$-field on $\OO$ must not be too rich for such a probability kernel to exist. Let us give a couple of definitions before stating at Proposition \ref{res-LZ06} that $R(\cdot\mid\mathcal{A})$ exists in a general setting.

We are looking for a conditional probability measure in the following sense.

\begin{definition}[Regular conditional probability kernel]
The kernel $R(\cdot\mid\mathcal{A})$ is a regular conditional probability if
\begin{enumerate}[(a)]
    \item for any $f\in L^1(R),$ $E_R(f\mid\mathcal{A})(\eta)=\IO
    f\,R(d\omega\mid\mathcal{A})(\eta)$ for $R_\mathcal{A}$-almost every $\eta;$
    \item for $R_\mathcal{A}$-almost every $\eta,$ $R(\cdot\mid\mathcal{A})(\eta)$ is a probability measure on $\Omega.$
\end{enumerate}
\end{definition}

Property (a) was proved at Remark\ref{rem-LZ01}
when $R$ is a bounded measure. It si property (b) which requires additional work, even when $R$  is bounded. 
Proposition \ref{res-LZ06} provides us with a general setting where such a regular kernel exists.
When a regular conditional kernel $R(\cdot\mid\mathcal{A})$
exists, \eqref{eq-LZ08} is concisely expressed as a disintegration formula:
\begin{equation}\label{eq-LZ04}
    R(d\omega)=\int_{\{\eta\in\Omega\}} R(d\omega\mid\mathcal{A})(\eta)\,R_\mathcal{A}(d\eta)
\end{equation}

\begin{definition} Let $\phi:\Omega\to\XX$ be a measurable function with values in a measurable space $\XX.$
The smallest sub-$\sigma$-field on $\Omega$ which makes $\phi$ a
measurable function is called the \emph{$\sigma$-field generated by $\phi.$} It is denoted by $\mathcal{A}(\phi).$
\end{definition}

We are going to consider the conditional expectation with respect to  $\mathcal{A}(\phi)$ which is denoted by
$$
E(\cdot\mid\mathcal{A}(\phi))=E(\cdot\mid\phi).
$$

\begin{proposition} Let $\mathcal{B}$ be the $\sigma$-field on $\XX$ and $\phi^{-1}(\mathcal{B}):=\{\phi^{-1}(B);B\in\mathcal{B}\}.$
\begin{enumerate}
    \item $\mathcal{A}(\phi)=\phi^{-1}(\mathcal{B}).$
    \item Any  $\mathcal{A}(\phi)$-measurable function $g:\Omega\to\RR$ can be written as $$g=\tilde g\circ\phi$$ with $\tilde g:\XX\to\RR$ a measurable function.
\end{enumerate}
\end{proposition}

\begin{proof}
    \boulette{(1)}
First remark that $\mathcal{A}(\phi)$ is the smallest sub-$\sigma$-field on $\Omega$ which makes $\phi$ a
measurable function. Consequently, it is the $\sigma$-field which is generated by $\phi^{-1}(\mathcal{B}).$
But it is easy to check that $\phi^{-1}(\mathcal{B})$ is a $\sigma$-field. Hence, $\mathcal{A}(\phi)=\phi^{-1}(\mathcal{B}).$

    \Boulette{(2)}
Let $y\in g(\Omega).$ As $g$ is $\mathcal{A}(\phi)$-measurable,
$g^{-1}(y)\in\mathcal{A}(\phi).$ By (1), it follows that there exists a measurable subset
$B_y\subset\XX$ such that $\phi^{-1}(B_y)= g^{-1}(y).$ Let us set
$$
\tilde{g}(x)=y,\quad \textrm{for all }x\in B_y.
$$
For any $\omega\in g^{-1}(y),$ we have $\phi(\omega)\in B_y,$
so that $g(\omega)=y=\tilde{g}(\phi(\omega)).$ But
$(g^{-1}(y))_{y\in g(\Omega)}$ is a partition of $\Omega,$ hence
$g(\omega)=\tilde{g}(\phi(\omega))$ for all $\omega\in\Omega.$
\end{proof}

This proposition allows us to denote
$$
x\in\XX\mapsto E_R(f\mid\phi=x)\in \RR
$$
 the unique function in $L^1(\XX,R_\phi)$ such
$E_R(f\mid\phi=\phi(\eta))=E_R(f\mid\mathcal{A}(\phi))(\eta),$ $R\ae$ en
$\eta.$

\begin{proposition}\label{res-LZ06}
Let $R\in\MO$ be a bounded positive measure on $\Omega$ and
$\phi:\Omega\to\XX$ a measurable  application in the Polish (separable, complete metric) space $\XX$ equipped with the corresponding Borel $\sigma$-field.  Then, $E_R(\cdot\mid\phi)$ admits a regular conditional probability kernel $x\in\XX\mapsto
R(\cdot\mid\phi=x)\in\PO.$ 
\end{proposition}

\begin{proof}
This well-known and technically delicate result can be found at
\cite[Thm 10.2.2]{Dud02}.
\end{proof}

In the setting of Proposition \ref{res-LZ06}, the disintegration formula \eqref{eq-LZ04} is
\begin{equation*}
    R(d\omega)=\int_{\XX} R(d\omega\mid\phi=x)\,R_\phi(dx).
\end{equation*}

The main assumption for defining properly $E_R(f\mid\mathcal{A})$ with $f\in L^1(R)$ at
Definition \ref{def-LZ08} is that $R_\mathcal{A}$ is $\sigma$-finite. In the special case where $\mathcal{A}=\mathcal{A}(\phi),$ it is equivalent to the following.

\begin{assumption}\label{hyp-01}
The measure $R_\phi\in\MX$ is $\sigma$-finite.
\end{assumption}

\begin{remark}[About this assumption]\label{rem-LZ05}
It is necessary that $R$ is $\sigma$-finite for $R_\phi$ to be  $\sigma$-finite too. Indeed, if $\seq{\XX}n$ is a $\sigma$-finite partition of $R_\phi,$ $(\phi^{-1}(\XX_n))_{n\ge1}$ is a countable  measurable partition of $\Omega$ which satisfies
$R(\phi^{-1}(\XX_n))=R_\phi(\XX_n)<\infty$ for all $n.$
This means that it finitely partitions $R$.
\end{remark}

\subsection*{Radon-Nikodym derivative and conditioning}

In addition to the measurable mapping
$\phi:\Omega\to\XX$ and the positive measure $R\in\MO,$ let us introduce another positive measure $P\in\MO$ which admits a Radon-Nikodym derivative with respect to  $R:$ $P\prec R.$

\begin{proposition}\label{res-LZ04} Under the Assumption \ref{hyp-01}, let us suppose that  $P$ is bounded and $P\prec R.$ Then, 
\begin{enumerate}
    \item We have $P_\phi\prec R_\phi$ and
$$\frac{dP_\phi}{dR_\phi}(\phi)=E_R\left(\frac{dP}{dR}\mid \phi\right),\quad
R\ae$$
    \item For any bounded measurable function $f,$
    $$
E_P(f|\phi)E_R\left(\frac{dP}{dR}\mid \phi\right) =E_R\left(\frac{dP}{dR} f\mid
\phi\right),\quad R\ae
    $$
    \item Furthermore, $$E_R\left(\frac{dP}{dR}\mid \phi\right)>0, \quad P\ae$$
\end{enumerate}
\end{proposition}

\begin{remark}
One might not have $E_R(\frac{dP}{dR}|\phi)>0,$  $R\ae$
\end{remark}

\begin{proof}
As $P$ is bounded, we have
    \begin{equation*}
    \frac{dP}{dR}\in L^1(R)
    \end{equation*}
and we are allowed to consider $E_R(\frac{dP}{dR}
f\mid\phi)$ for any bounded measurable function $f.$

 \Boulette{(1)}
For any bounded measurable function $u$ on $\XX,$
\begin{eqnarray*}
  E_{P_\phi}(u)&=&E_P(u(\phi))=E_R\left(\frac{dP}{dR}
u(\phi)\right) \\
  &=& E_R\left(u(\phi)
E_R\left(\frac{dP}{dR}\mid\phi\right)\right)
    =E_{R_\phi}\left(u
E_R\left(\frac{dP}{dR}\mid\phi=\cdot\right)\right)
\end{eqnarray*}

    \Boulette{(2)}
For any bounded measurable functions  $f,h$ with
$h\in\mathcal{A}(\phi),$ we ahve
\begin{eqnarray*}
  E_P(h f)&=&E_R\left(\frac{dP}{dR} h f\right)=E_R\left(h  E_R\left(\frac{dP}{dR} f\mid
\phi\right)\right)\quad \textrm{and} \\
  E_P(h f)&=&E_P(h E_P(f\mid\phi))=E_R\left(h
E_P\left(f\mid\phi\right)\frac{dP}{dR}\right)=E_R\left[h
E_P(f\mid\phi)E_R\left(\frac{dP}{dR}\mid\phi\right)\right].
\end{eqnarray*}
The desired result follows by identifying the  right-hand side terms of these series of equalities.

\Boulette{(3)} Let $A\in\mathcal{A}(\phi)$ be such that $\1_A
E_R\left(\frac{dP}{dR}\mid\phi\right)=0,$ $R\ae$ Then,\\ $0=E_R\left(\1_A
E_R\left(\frac{dP}{dR}\mid\phi\right)\right)=E_R\left(\frac{dP}{dR}
\1_A\right)=P(A).$ This proves the desired result.
\end{proof}


\end{document}